\documentclass[11pt]{article}

\usepackage{amsmath,amssymb}
\usepackage{amsthm}
\usepackage{inputenc}

\usepackage{mathrsfs}
\usepackage[usenames]{color}
\usepackage{graphicx} 

\usepackage{amsthm,amsfonts,amsmath,verbatim,amssymb,verbatim,cite}

\usepackage{mathrsfs}
\usepackage[usenames]{color}
\usepackage{graphicx}
\usepackage{subfig}
\numberwithin{equation}{section}
\newtheorem{theorem}{Theorem}[section]
\newtheorem{prop}[theorem]{Proposition}
\newtheorem{result}[theorem]{Result}
\newtheorem{lemma}[theorem]{Lemma}
\newtheorem{definition}[theorem]{Definition}
\newtheorem{claim}[theorem]{Claim}
\newtheorem{cor}[theorem]{Corollary}
\newtheorem{conj}[theorem]{Conjecture}

\newtheorem{problem}[theorem]{Problem}

\newtheorem{remark}[theorem]{Remark}

\begin{document}




\title{Transversals in generalized Latin squares}


\author{J\'anos Bar\'at\\
\small University of Pannonia, Department of Mathematics\\[-0.8ex]
\small 8200 Veszpr\'em, Egyetem utca 10., Hungary\\[-0.8ex]
\small and\\
\small  MTA-ELTE Geometric and Algebraic Combinatorics Research Group\\[-0.8ex]
\small H--1117 Budapest, P\'azm\'any P.\ s\'et\'any 1/C, Hungary\\[-0.8ex]
\small \texttt{barat@cs.elte.hu}\\
and\\
Zolt\'an L\'or\'ant Nagy \thanks{Research supported  by
OTKA Grant No.  K 120154}\\
\small MTA--ELTE Geometric and Algebraic Combinatorics Research Group\\[-0.8ex]
\small H--1117 Budapest, P\'azm\'any P.\ s\'et\'any 1/C, Hungary\\[-0.8ex]
\small \texttt{nagyzoli@cs.elte.hu}
}

\date{}

\maketitle

\begin{abstract} 
We are seeking a sufficient condition that forces a transversal in a generalized Latin square.
A generalized Latin square of order $n$ is equivalent to a proper edge-coloring of $K_{n,n}$.
A transversal corresponds to a multicolored perfect matching.
Akbari and Alipour defined  $l(n)$ as the least integer such that every properly edge-colored $K_{n,n}$, 
which contains at least $l(n)$ different colors, admits a multicolored perfect matching.
They conjectured that $l(n)\leq n^2/2$ if $n$ is large enough.
In this note we prove that $l(n)$ is bounded from above by $0.75n^2$ if $n>1$.
We point out a connection to anti-Ramsey problems.
We propose a conjecture related to a well-known result by Woolbright and Fu, that
every proper edge-coloring of $K_{2n}$ admits a multicolored $1$-factor.
\end{abstract}


\section{Multicolored matchings and generalized Latin squares}

A subgraph $H$ of an edge-colored host graph $G$ is \textit{multicolored} if the edges of $H$ has different colors.
The study of \textit{multicolored} (also called \textit{rainbow, heterochromatic}) subgraphs
dates back to the 1960's.
However, the special case of finding multicolored perfect matchings in complete bipartite graphs was first studied much earlier by 
Euler in the language of Latin squares. 
Since then this branch of combinatorics, especially the mentioned special case, has been flourishing. 
Several excellent surveys were dedicated to the subject, see \cite{Wanless, Kao-Li, Laywine-Mullen, Fujita}.

In this paper we mainly focus on the case when the host graph is a complete bipartite graph $K_{n,n}$, and the multicolored subgraph in view is a perfect matching (\textit{1-factor}).
There is a natural constraint on the coloring: it has to be proper. \\
These conditions can be reformulated  in the language of Latin squares. 
A \textit{Latin square} of order $n$ is an $n\times n$ matrix, which has $n$ different symbols as entries, and each symbol appears exactly once in each row and in each column. 
A \textit{generalized} Latin square of order $n$ is an $n\times n$ matrix, in which each symbol appears at most once in each row and in each column. 
A  \textit{diagonal} of a generalized Latin square of order $n$ is a set of entries, which contains exactly one representative from each row and column.  
If the symbols are all different in a diagonal, then we call it a \textit{transversal}.\\
Generalized Latin squares correspond to properly edge-colored complete bipartite graphs, while transversals 
correspond to multicolored 1-factors (perfect matchings).  
The so called \textit{partial transversals} correspond to multicolored matchings.  
This intimate relation allows us to use the concept of symbols and colors interchangeably.

It is known that there exist Latin squares without a transversal.
One might think that using more symbols should help finding a transversal.
Therefore, it is natural to seek the sufficient number of symbols.
We recall the following

\begin{definition}[Akbari and Alipour \cite{Akbari}] 
Let $l(n)$ be the least number of symbols that forces a transversal in any generalized Latin square of order $n$ that contains at least $l(n)$ symbols. 
\end{definition}

In the terminology of matchings, they asked the threshold  for the number $l(n)$ of colors such that any proper $l$-coloring of $K_{n,n}$ contains a multicolored perfect matching if $l\geq l(n)$. 
Notice that the function $l(n)$ is not obviously monotone increasing. 


Akbari and Alipour determined $l(n)$ for small $n$: $l(1)=1, l(2)=l(3)=3, l(4)=6$. 
They  also proved that $l(n)\geq n+3$ for $n=2^a-2$ ($2<a\in \mathbb{N}$).
They posed the following

\begin{conj}[Akbari and Alipour \cite{Akbari}]\label{sejtes}
 $l(n)-n$ is not bounded if $n\rightarrow \infty$, while $l(n)\leq n^2/2$ if $n>2$.
\end{conj}

Our main contribution is the following theorem.

\begin{theorem}\label{main} $l(n)\leq 0.75n^2$ if $n>1$.
\end{theorem}

Although we conjecture that $l(n)= o(n^2)$, we must mention that if we relax the settings by allowing symbols to appear more than once in the columns, 
then there exist $n\times n$ transversal-free matrices for all $n$ which contain $n^2/2+O(n)$ symbols \cite{Barat}.

The paper is built up as follows. 
In Section~\ref{2} we show the connection of the problem to a classical Erd\H{o}s--Spencer result.
We prove an upper bound on $l(n)$ using a refined variant of the Lov\'asz local lemma. 
We present the proof of Theorem~\ref{main} which is mainly built  on K\"onig's theorem. 
Finally, in Section~\ref{3}, we propose the study of a function similar to $l(n)$, and investigate the relation to certain Anti-Ramsey problems. 

\section{Two approaches to bound the number of symbols}\label{2}

\subsection{Lov\'asz local lemma}

It is a classical application of the Lov\'asz local lemma (LLL) that there exists a transversal in  an $n\times n$ matrix if no color appears more than $\frac{1}{4e}n$ times.
In fact, Erd\H os and Spencer \cite{ES} weakened the conditions of LLL by introducing  the so called lopsided dependency graph $G$ of the events, 
on which the following holds for every event $E_i$ and every subfamily $\mathcal{F}$ of events $\{ E_j: j\not \in N_G[i]\}$:

 $$P(E_i \mid \cap_{j\in \mathcal{F}} \overline{E_j}) \leq P(E_i),$$
where  $N_G[i]$ denotes the closed neighborhood of vertex $i$ in  graph $G$.
Under this assumption, it is enough to show the existence of an assignment $i\longmapsto (\mu_i>0)$ which fulfill \begin{equation}\label{eq:cond}P(E_i) \leq \frac{\mu_i}{\sum_{S\subseteq N_G[i]} \prod_{j\in S} \mu_j}\end{equation} to obtain $P(\cap_i \overline{E_i})>0$.

Applying the ideas of Scott and Sokal \cite{SS}; Bissacot, Fern\'andez,  Procacci and  Scoppola \cite{LLLextra} observed that  LLL remains valid if the summation in Inequality \ref{eq:cond} is
restricted to those sets $S$ which are independent in $G$. 

\bigskip

Let $c(a_{ij})$ denote the number of occurrences of the symbol $a_{ij}$ in an $n\times n$ array $A$ \ $(n>1)$. 
Let $c_{i*}(A)$ and $c_{*j}(A)$ measure the average occurrence in row $i$ and  column $j$ as   
$$c_{i*}(A)=\left(\sum_t c(a_{it})\right) -n\ \text{and}\ c_{*j}(A)=\left(\sum_t c(a_{tj})\right) -n.$$ 

It can be viewed as some kind of weight-function on the rows and columns, where the weight is zero only if all entries admit uniquely occurring colors.

We follow the proof of the improvement on the Erd\H os-Spencer result in  \cite{LLLextra}. 
We show that  $P(\cap_v \overline{E_v})>0$ holds for the set of events $\{E_v\}$ that a random diagonal contains a particular pair $v$ of monochromatic entries. 
Here $|N_G[v]|$ in the lopsided dependency graph $G$ depends only on the number of monochromatic pairs $(v, v')$ of entries, 
which shares (at least) one row or column with an entry from both $v$ and $v'$. 
Thus if $v$ consists of $a_{ij}$ and $a_{kl}$, then $|N_G[v]|\leq c_{i*}(A)+c_{*j}(A)+c_{k*}(A)+c_{*l}(A)$. 
Also if $w, w'\in N_G[v]$ covers the same row from $\{i,k\}$ or column from $\{j,l\}$ then $w$ and $w'$ are adjacent in $G$.
 
If we set $\mu_v:=\mu \  \forall v$, then it is enough to provide a $\mu$ such that 
 $$ P(E_v)= \frac{1}{n(n-1)} \leq \frac{\mu_v}{\sum\limits_{S\subseteq N_G[v], S \mbox{\ indep.}} \prod\limits_{j\in S} \mu_j} = \frac{\mu}{\sum\limits_{S\subseteq N_G[v],\  S \mbox{ \ indep.}}  \mu^{|S|}}$$ 

Consequently, it is enough to set $\mu$ in such a way that   $$\frac{\mu}{\sum\limits_{S\subseteq N_G[v],\  S \mbox{ \ indep.}}  \mu^{|S|}}> \frac{\mu}{(1+ c_{i*}(A)\mu)(1+ c_{*j}(A)\mu)(1+c_{k*}(A)\mu)(1+ c_{*l}(A)\mu)}\geq  \frac{1}{n(n-1)}$$ holds.
	
It is easy to see that $(1+U\mu)(1+V\mu)\leq(1+\frac{U+V}{2}\mu)^2$ for all $U, V\in \mathbb{R}$, hence   	$$\frac{\mu}{(1+\overline{c_v}\mu)^4} \geq  \frac{1}{n(n-1)}$$ implies the required condition, where $$\overline{c_v}:= \frac{c_{i*}(A)+c_{*j}(A)+c_{k*}(A)+c_{*l}(A)}{4}.$$
	
Thus if  we set $\mu:= \frac{1}{3\overline{c_v}}$, we obtain the following 

\begin{prop}\label{key} There always exists a transversal in a generalized Latin square unless \begin{equation}\label{ineq}
	\left(\frac{4}{3}\right)^3(c_{i*}(A)+c_{*j}(A)+c_{k*}(A)+c_{*l}(A))>n(n-1)\end{equation}
	
	 holds for a pair of monochromatic entries $a_{ij}$ and $a_{kl}$.
\end{prop}	

\begin{cor} $l(n)\leq (1-\frac{27}{256}) n^2+\frac{27}{256}n \approx 0.895n^2$	if $n>1$.
\end{cor}	
	
\begin{proof} 	Observe that $n^2-c_{i*}(A)$ or   $n^2-c_{*j}(A)$ bounds from above the number of colors in $A$  for all $  i, j\in [1,n] $. Consequently, if the number of colors is at least $(1-\frac{27}{256}) n^2+\frac{27}{256}n $, then
$$\left(\frac{4}{3}\right)^3c_{i*}(A)\leq \frac{1}{4}(n^2-n) {\mbox{\ \ and \ \ }} \left(\frac{4}{3}\right)^3c_{*j}(A)\leq \frac{1}{4}(n^2-n)$$ for every row $i$ and column $j$, which in turn provides the existence of a transversal according to Proposition \ref{key}.
\end{proof}	
	
	
\begin{remark}\label{notabene}  Note that while the proof of Erd\H os and Spencer points out the existence of one frequently occurring symbol, the proof above reveals that in fact many symbols must occur frequently to avoid a transversal.
\end{remark}	
	
\bigskip
\bigskip

\subsection{Using K\"onig's theorem}

We start with a lemma on the structure of partial transversals, which is essentially the consequence of the 
greedy algorithm. The following easy observation is due to Stein \cite{Ste75}. 

\begin{result}\label{base} Consider $r$ rows in a generalized Latin square $A$ of order $n$. 
If $\frac{n+1}{2}\geq r$, then there exists a partial transversal of order $r$ in $A$ covering the $r$ rows in view.
\end{result}

\begin{lemma} \label{fullrows} Consider $p$ rows and $q$ columns in an $n\times n$ generalized Latin square.
 If $q \leq p \leq (n+1)/2$, then either\\
(Case (a)) $q\leq p/2$ and there exists a partial transversal of size $p$ covering the $p$ rows and $q$ columns, or\\
(Case (b))    $q>p/2$ and there exists a partial transversal of size $\lfloor p/2\rfloor+q$ covering the $p$ rows and $q$ columns.
\end{lemma}

\begin{proof} Both parts follow from the fact that we can choose $\min\{q,\lfloor p/2\rfloor\}$ entries in the array formed by the intersection of the $p$ rows and $q$ columns 
and we can complete it greedily by Result~\ref{base}.
\end{proof}

We proceed by recalling a variant of K\"onig's theorem, see Brualdi, Ryser \cite{BR}.

\begin{lemma}\label{konig} 
There exists an all-$1$ diagonal in a $0/1$ square matrix of order $n$ if and only if 
there does not exist an all-$0$ submatrix of size $x\times y$, where $x+y\ge n+1$. 
\end{lemma}

Now we prove another upper bound on $l(n)$.

\begin{theorem} If a generalized Latin square of order $n$ contains at least $0.75n^2$ symbols, then it has a transversal.
\end{theorem}

\begin{proof}


First notice that the statement holds for $n=1, 2$. 
We proceed by induction. 
Consider a generalized Latin square $A$ of order $n$, which contains at least $0.75n^2$ symbols. 
A symbol is a \textit{singleton} if it appears exactly once in $A$. 
We refer to the other symbols as \textit{repetitions}. 
A submatrix is called a \textit{singleton-}, resp. \textit{repetition-submatrix} if every entry of the matrix is a singleton, resp. repetition.

Let $p$ be the number of rows consisting only of repetitions and $q$ be the number of  columns consisting only of repetitions. 
We refer to these as full rows and columns, and assume that $q\leq p$.  
Notice that  $p\leq n/2$, since the number of symbols is at least $0.75n^2$.  \\
Our aim is to choose a partial transversal that covers all full rows and columns, and then we complete this to a transversal by adding only singletons.
First we apply Lemma~\ref{fullrows} to get a partial transversal that covers the full rows and columns. 
Next, we omit the rows and columns that are covered by the chosen partial transversal. 
We obtain a generalized Latin square $A'$ of order $(n-p)$ in Case (a) or of order $(n-p/2-q)$ in Case (b). 
Now we are done by Theorem~\ref{konig}, if there are no $x\times y$ repetition-submatrices in $A'$ of order $(n-p)$ in Case (a),  where $x+y>n-p$, or 
there are no $x\times y$ repetition-submatrices in $A'$ of order $(n-p/2-q)$  in Case (b), where $x+y>n-p/2-q$.   

We suppose to the contrary that such a repetition-submatrix exists in one of the cases.
Note first that in either case, $A'$ does not contain full rows and columns. 
Therefore, we can choose a singleton $\sigma_1$ in $A'$ such that  at least $x$  repetitions appear in its row.
Similarly,  we can choose a singleton $\sigma_2$ in $A'$ such that  at least $y$  repetitions appear in its column. 

\begin{claim}\label{repeta} 
There exists a singleton $\sigma$ whose row or column contains more than $n/2$ repetitions in the original Latin square $A$ both in Case (a) and (b).
\end{claim}
\begin{proof}
In Case (a) $q\leq p/2$. \\
The number of repetitions in the row  of $\sigma_1$ is at least $q+x$ and  number of repetitions in the  column of $\sigma_2$ is at least
  $p+y$. Thus the statement holds since $p+q+x+y> p+q+(n-p)\geq n$.

In Case (b), $q> p/2$. \\
The number of repetitions in the row  of $\sigma_1$ is at least $q+x$ and  number of repetitions in the column of $\sigma_2$ is at least
  $p+y$. Thus the statement holds since $p+q+x+y> p+q+(n-p/2-q)\geq n$.
\end{proof}

In view of Claim \ref{repeta}, if we omit the row and column of the singleton $\sigma$, we obtain a generalized Latin square $B$ of order $n-1$, 
which admits more than $0.75n^2-(2n-1)+ n/2> 0.75(n-1)^2$ symbols. 
By the induction hypothesis, there exists a transversal in $B$, hence it can be completed to a transversal of $A$ by adding $\sigma$.
\end{proof}


\section{Discussion}\label{3}

At the time of submission, we learned that Best, Hendrey, Wanless, Wilson and Wood \cite{uj} achieved results similar to ours. 
As the best upper bound, they proved $l(n)<(2-\sqrt{2})n^2$. 
Nevertheless, not only the conjecture of Akbari and Alipour remained open, but it is plausible that it can be strengthened in the order of magnitude as well. 
In fact, the bound $\frac{1}{2}n^2$ is intimately related to the number of singletons, which took a crucial part in both our proof and the proof in  \cite{uj}. 
If the number of colors does not exceed $\frac{1}{2}n^2$, then there might be no singletons at all. 
However, our first probabilistic proof implies also that either there exists a transversal in a generalized Latin square of order $n$ with $Cn^2$ colors ($C>0.45$), 
or the number of singletons is large. This fact points out that the constant $1/2$ in Conjecture~\ref{sejtes} is highly unlikely to be sharp. 
More precisely, we show the following

\begin{prop} If the number of singletons is less than  $(2C+0.25\left(\frac{3}{4}\right)^3-1+o(1))n^2$ in a generalized Latin square of order $n$  with $Cn^2$ symbols, 
then there exists a transversal.
\end{prop}

\begin{proof} Suppose first that in every row and column, the sum $c_{i*}(A)$ and $c_{*j}(A)$ are below  $0.25\left(\frac{3}{4}\right)^3(n^2-n)$. 
 This in turn implies the existence of a transversal by Proposition~\ref{key}. 
 On the other hand, if for example $c_{i*}(A)$ exceeds that bound, then consider only the symbols not appearing in row $i$, and 
 let us denote by $n_k$ the number of occurrences of symbol $k$, which does not occur in row $i$. 
 Clearly $\sum_k n_k = Cn^2-n$ and $\sum_k kn_k=(n(n-1)-c_{i*}(A))\le(1-0.25\left(\frac{3}{4}\right)^3)(n^2-n)$. 
 Consequently, for the number of singletons not appearing in the $i$th row,
 $$|\{t: n_t=1 \}|\geq 2\sum_k n_k -\sum_k kn_k\ge (2C+0.25\left(\frac{3}{4}\right)^3-1+o(1))n^2,$$ which makes this case impossible.
\end{proof}

A special case, that appears as a bottleneck in some arguments concerns generalised Latin squares, where 
each repeated symbol has maximum multiplicity.
We show that also in this special case one can find a transversal.

\begin{lemma}
 If $A$ is a generalised Latin square of order $n$, where each symbol has multiplicity $1$ or $n$ (and both multiplicities occur), 
 then $A$ has a transversal.
\end{lemma}

\noindent {\bf Proof.}
We associate an edge-colored complete bipartite graph $G_A$ to $A$ such that vertices on one side correspond to rows 
the other side to columns and the colors of the edges to the symbols.
Our goal is to find a multicolored matching.

Notice that the Latin property implies that a symbol with multiplicity $n$ corresponds to a perfect matching.
Let us remove all edges corresponding to symbols with multiplicity $n$.
If there are $r$ such colors, then the remaining bipartite graph is $(n-r)$-regular.
As an easy corollary of Hall's theorem, any regular bipartite graph contains a perfect matching.
In our case there are only singleton colors on the edges, so the perfect matching is multicolored.
$\Box$

It seems likely that if the number of colors is large, then we not only obtain one transversal, but also a set of disjoint transversals. 
This motivates the study of the following function.

\begin{definition} Let $l^{*}(n)$ be the least integer such that for any proper edge-coloring of $K_{n,n}$ by at least $l^*(n)$ colors, 
the colored graph can be decomposed into the disjoint union of $n$ multicolored perfect matchings.
\end{definition}

\begin{conj}  $l^{*}(n)\leq n^2/2$ if $n$ is large enough.
\end{conj}

We remark that the difference of $l(n)$ and $l^{*}(n)$ is at least linear in $n$.

\begin{prop} $l^{*}(n)-l(n)\geq n-1$. 
\end{prop}

\begin{proof} For $n\leq 2$ the claim is straightforward. 
Suppose $n\geq 3$. 
By definition, there exists a transversal-free generalized Latin square $A$ of order $n$ with $l(n)-1$ symbols. 
Since $l(n)\leq 0.75n^2$, we can find a set $S$ of $n-1$ different repetitions, where $n-1\leq 0.25n^2$.
We assign new symbols to the entries of $S$ to create a new generalized Latin square $A'$ of the same order. 
Since $S$ cannot cover $n$ disjoint transversals, and there were no transversals disjoint to $S$, matrix $A'$ cannot be decomposed to $n$ transversals, but contains
 $l(n)+n-2$ symbols.
\end{proof}

The question we studied concerning $l(n)$  clearly has an anti-Ramsey flavor. 
The anti-Ramsey number $AR(n, \mathcal{G})$ for a graph family  $\mathcal{G}$, introduced
by Erd\H{o}s, Simonovits and S\'os \cite{ESS}, is the maximum number of colors in an edge coloring of $K_n$ that
has no multicolored (rainbow) copy of any graph in $\mathcal{G}$. To emphasize this connection, we propose the following problem.

\begin{problem}\label{anti} 
What is the least number of colors $t(n,2)$, which guarantees a rainbow $2$-factor subgraph on at least $n-1$ vertices in a 
properly edge-colored complete graph $K_n$ colored by at least $t(n,2)$ colors?
\end{problem}

Perhaps the size $n-1$ of the $2$-factor subgraph seems artificial in some sense at first, or at least it could be generalized to any given function $f(n)$.
We recall that for the function $t(n,1)$ corresponding to  $1$-factors,   Woolbright and Fu provided the following related result.
In Problem~\ref{anti}, we have to allow two values $n-1$ and $n$ to avoid parity issues. 

\begin{prop} \cite{Fu} Every properly colored $K_{2n}$ has a multicolored $1$-factor if the number of colors is at least $2n-1$ and $n>2$. That is, $t(n,1)=n-1$.
\end{prop}

In another formulation, the necessary number of colors for a proper edge-coloring is already sufficient to guarantee a multicolored perfect matching.
It might happen that it also forces a much larger structure as required in Problem~\ref{anti}.
We propose the following

\begin{conj}
 Any proper edge-coloring of $K_{2n}$ by $2n-1$ colors contains a multicolored $2$-factor on $2n-1$ or $2n$ vertices.
\end{conj}

If the above conjecture fails, then possibly there are proper edge-colorings of $K_n$ without multicolored  $2$-factors of size $n$ or $n-1$. 
In that case, we can use a connection between $t(n,2)$ and $l(n)$ to show a lower bound.

\begin{prop} $l(n)\geq t(n,2)+1$.
\end{prop}

\begin{proof} Consider an edge-coloring $C$ of the complete graph $K_n$ on vertex set $V$ without multicolored $2$-factors of size $n$ or $n-1$. 
We associate to $C$ a coloring of the complete bipartite graph $K_{n,n}$ on partite classes $U$ and $W$ as follows:  
let us assign the color of $v_iv_j\in E(K_n)$ ($i,j\in [1,n]$) to the edge $u_iw_j \in E(K_{n,n})$ if $i\neq j$, 
and color the set of independent edges $u_iw_i$ ($i\in [1,n]$) by a separate color.
Suppose that we found a multicolored $1$-factor $M$ in the complete bipartite graph. 
We omit at most one  edge of $M$ if we delete the edges $u_iv_i$ and $M'$ remains. 
Consider the edges $v_kv_l$ in  $K_n$,  for which  $u_kv_l$ is contained in the multicolored $M'$ of edges. 
This edge set is multicolored too, and each vertex has degree $2$.
\end{proof}

\end{document}